\documentclass{amsart}
\usepackage{hyperref}
\usepackage{fullpage}
\usepackage{amsrefs}
\usepackage{tikz}
\usepackage{dsfont}
\usepackage{newverbs}
\usepackage{fancyvrb}
\usepackage{comment}
\usepackage{mathrsfs}
\usepackage{mathtools}
\usepackage{amssymb}
\usepackage{stmaryrd}

\usepackage{easytable, pdflscape, enumitem, euscript, amsmath}

\usepackage{bbm}
\usepackage{array}

\usepackage{tikz}
\usetikzlibrary{matrix,arrows,decorations.pathmorphing, cd}

\newif\ifshowkeys
\showkeysfalse

\ifshowkeys

\else

\fi

\newcommand{\CT}        {{\mathcal{T}}}

\newcommand{\fp}        {{\mathfrak{p}}}

\newcommand{\bc}[1]	{\langle #1\rangle}

\newcommand{\mbu}         {\mathbbm{1}}

\newcommand{\colim} {\operatornamewithlimits{\underset{\longrightarrow}{lim}}}

\newcommand{\hocolim}  {\operatornamewithlimits{\underset{\longrightarrow}{holim}}}

\newtheorem{theorem}{Theorem}[section]

\newtheorem{lemma}[theorem]{Lemma}
\newtheorem{proposition}[theorem]{Proposition}
\newtheorem{corollary}[theorem]{Corollary}
\theoremstyle{definition}
\newtheorem{remark}[theorem]{Remark}
\newtheorem{definition}[theorem]{Definition}
\newtheorem{example}[theorem]{Example}

\theoremstyle{theorem}
\newtheorem*{theorem*}{Theorem}
\newtheorem*{question}{Question}

\newcommand\blfootnote[1]{%
  \begin{NoHyper}%
  \renewcommand\thefootnote{}\footnote{#1}%
  \addtocounter{footnote}{-1}%
  \end{NoHyper}%
}

\title{The local-to-global principle via topological properties of the tensor triangular support}
\author{Nicola~Bellumat}

\begin{document}
\blfootnote{This research was partially supported by grant no.~DNRF156 from the Danish National Research Foundation.}
\maketitle

\begin{abstract}
Following the theory of tensor triangular support introduced by Sanders, which generalizes the Balmer-Favi support, we prove the local version of the result of Zou that the Balmer spectrum being Hochster weakly scattered implies the local-to-global principle.  

That is, given an object $t$ of a tensor triangulated category $\CT$  we show that if the tensor triangular support $\text{Supp}(t)$ is a weakly scattered subset with respect to the inverse topology of the Balmer spectrum $\text{Spc}(\CT^c)$, then the local-to-global principle holds for $t$. 

As immediate consequences, we have the analogue adaptations of the well-known statements that the Balmer spectrum being noetherian or Hausdorff scattered implies the local-to-global principle.

We conclude with an application of the last result to the examination of the support of injective superdecomposable modules in the derived category of an absolutely flat ring which is not semi-artinian.
\end{abstract}

\section{Introduction}
Our starting point is the following question, raised in \cite{BCHS-costratification}:
\begin{question}[{\cite[Question~21.8]{BCHS-costratification}}] \label{question-noethsupp}
Let $t$ be an object of a rigidly-compactly generated tensor triangulated category $\CT$. Suppose that the Balmer spectrum $\emph{Spc}(\CT^c)$ is weakly noetherian and the Balmer-Favi support $\emph{Supp}(t)$ is a noetherian subspace of $\emph{Spc}(\CT^c)$. Then, does the local-to-global principle hold for $t$?
\end{question}
This problem was motivated by \cite[Thm.~3.21]{BHS-stratification}, which states that in the case the whole Balmer spectrum is noetherian then the local-to-global principle holds. This result is nothing else that the adaptation to the setting of stratification via Balmer-Favi support presented in \cite{BHS-stratification} of the analogous finding by Benson, Iyengar and Krause demonstrated in \cite[Thm.~3.6]{bik}.

However, there is a statement even stronger than \cite[Thm.~3.21]{BHS-stratification}: in \cite[Thm.~7.6]{zou} Zou showed that if the Balmer spectrum $\text{Spc}(\CT^c)$ considered with the inverse topology is weakly scattered then the local-to-global principle holds. This work uses the tensor triangular support: a generalization of the Balmer-Favi support, developed by Sanders in \cite{supp-sanders}, which does not require the Balmer spectrum to be weakly noetherian.

Inspired by \cite[Thm.~7.6]{zou} and by the spirit of the Question, i.e.\ to weaken the assumption guaranteeing the local-to-global principle
to just the support of the examined object, we prove the following:
\begin{theorem*}
Let $t $ be an object of a rigidly-compactly generated tensor triangulated category $\CT$ such that its tensor triangular support $\emph{Supp}(t)$ is weakly scattered in $\emph{Spc}(\CT^c)$ endowed with the inverse topology. Then the local-to-global principle holds for $t$.
\end{theorem*}
This theorem gives a positive answer to \cite[Question~21.8]{BCHS-costratification}. Moreover, it also implies the relative statement of \cite[Thm.~5.6]{ste-dimfunctions}: if  $\text{Spc}(\CT^c)$ has the constructible topology and we have $t \in \CT$ with $\text{Supp}(t)$ scattered, then the local-to-global principle holds for $t$.

Finally, we use this last result to deduce that in $\text{D}(R)$, the derived category of a non semi-artinian absolutely flat ring, the support of an injective superdecomposable module must be contained in the maximal perfect subset of $\text{Spc}(\text{D}(R)^c)$, where this Balmer spectrum is homeomorphic to the Zariski spectrum $\text{Spec}(R)$.

We construct one of such modules so that its support coincides with the maximal perfect subset. The construction depends on the correspondence between the Loewy series of $R$ and the Cantor-Bendixson derivatives of the Zariski spectrum $\text{Spec}(R)$.\\

\textbf{Acknowledgements:} I thank Greg Stevenson for his interest and support to this work. I also thank Marcus Tressl for his advice on various matters regarding the topology of spectral spaces.

\section{Setup and notation}
We will mostly follow the notation and definitions of \cite{BCHS-costratification} and \cite{zou}. We will explicitly state when we introduce different notation.

From now on, $(\CT,\otimes, \mbu)$ denotes a rigid-compactly generated tensor triangulated category. Unless stated otherwise, we do not assume the Balmer spectrum $\text{Spc}(\CT^c)$ to be weakly noetherian (\cite[Def.~2.3]{BHS-stratification}). We will be considering on $\CT$ the tensor triangular support of \cite[Def.~5.7]{zou}, which was first introduced in \cite[Def.~4.1]{supp-sanders}. Since we will be using only this notion of support, from now on we will call it just support without further specification.

We make two assumptions on the tensor triangulated category $\CT$: first, we assume it satisfies the detection property (\cite[Def.~6.9]{zou}), i.e.\ if there is an object $t\in \CT$ such that $\text{Supp}(t)=\emptyset$ then $t=0$. 

Second, we assume $\CT$ admits a model, i.e.\ it is the homotopy category of a stable monoidal model category or of a stable monoidal infinity category, or it is the underlying category of a strong stable monoidal derivator. This allows us to have a well-behaved theory of homotopy colimits. See \cite[Rmk.~2.19]{ste-dimfunctions} for more details.

We introduce our own notation regarding the Rickard idempotents:
\begin{definition}
Let $Y\subseteq \text{Spc}(\CT^c)$ be a Thomason subset, so we have the corresponding thick tensor ideal of compact objects $\CT^c_{Y}$ by the classification theorem due to Balmer. Then we can invoke \cite[Thm.~3.3.5]{hopast:ash} to obtain from $\CT^c_{Y}$ an exact triangle
\[ \Gamma_{Y}\mathbbm{1}\rightarrow \mathbbm{1} \rightarrow L_{Y^c}\mathbbm{1}\rightarrow \Sigma \Gamma_Y\mathbbm{1} \]
where $\Gamma_{Y}\mathbbm{1} \in \text{Loc}\bc{\CT^c_Y}$ and $L_{Y^c}\mathbbm{1} \in \text{Loc}\bc{\CT^c_Y}^{\perp}$ are the idempotents associated respectively with the acyclicization and localization with respect to $\text{Loc}\bc{\CT^c_Y}$.

In the notation of \cite{BCHS-costratification}, we have the equalities
\[ e_{Y}=\Gamma_{Y}\mbu \qquad \qquad f_Y=L_{Y^c}\mbu. \]
We recall from \cite[Def.~2.2]{zou} that a subset $Z \subseteq \text{Spc}(\CT^c)$ is weakly visible if there exist two Thomason subsets $Y_1, Y_2 \subseteq \text{Spc}(\CT^c)$ such that $Z=Y_1\cap Y_2^c$. We denote by $g(Z)=\Gamma_{Y_1}\mbu \otimes L_{Y_2^c}\mbu$ the corresponding idempotent. If $\fp \in \text{Spc}(\CT^c)$ is a Balmer prime such that $\{\fp \}$ is weakly visible, then we will write simply $g(\fp)$ for the associated idempotent.
\end{definition}
We make this choice so that the subscript of an idempotent coincides with its support and to emphasize that left and right Rickard idempotents correspond to finite colocalizations and localizations respectively.

\begin{remark}\label{rmk-detectionidempotents}
It is proved in \cite[Thm.~4.2]{supp-sanders} that for any $Y\subseteq \text{Spc}(\CT^c)$ Thomason subset and $t\in \CT$ it holds
\[ \text{Supp}(\Gamma_Yt)=\text{Supp}(t)\cap Y \qquad \qquad \text{Supp}(L_{Y^c}t)=\text{Supp}(t)\cap Y^c. \]
Using these equalities and the detection property, we can see that if $\text{Supp}(t)\cap Y^c=\emptyset$ then $t\cong t\otimes \Gamma_Y \mbu$ and dually the equality $Y\cap \text{Supp}(t)=\emptyset$ implies $t\cong t\otimes L_{Y^c}\mbu$.
\end{remark}

\begin{definition}
Given a generic spectral space $X$ and a subset $V\subseteq X$ we denote by $\overline{V}^{\vee}$ the closure operator with respect to the inverse (i.e.\ Hochster dual) topology $X^{\vee}$. That is, $\overline{V}^{\vee}$ is the smallest complement of a Thomason subset of $X$ containing $V$. See \cite[Def.~1.4.1]{spespa} for a reference.

We denote by $X^{\flat}$ the set $X$ considered with the constructible (or patch) topology, see \cite[Def.~1.3.11]{spespa}. We denote by $\overline{V}^{\flat}$ the closure of $V$ in this topology.
\end{definition}

\begin{remark}
The condition of $\text{Spc}(\CT^c)$ being weakly noetherian means that its Hochster dual $\text{Spc}(\CT^c)^{\vee}$ is $T_D$, i.e.\ each of its points is locally closed (see \cite[\S4.5.10]{spespa}). If this property is satisfied, then for every prime $\fp \in \text{Spc}(\CT^c)$ we can define the idempotent $g(\fp)$ and the tensor triangular support coincides with the Balmer-Favi support examined in \cite{BHS-stratification} and \cite{BCHS-costratification}.
\end{remark}

\section{Support weakly scattered in the inverse topology}
We want to prove that an object $t \in \CT$ with the property that its support $\text{Supp}(t)$ is weakly scattered in $\text{Spc}(\CT^c)^{\vee}$ satisfies the local-to-global principle, in the sense of \cite[Def.~7.1]{zou}. The first step is showing that fixed a cover of the Balmer spectrum by weakly visible subsets there exists a refinement with the properties we need.
\begin{proposition}\label{prop-refinementcoverweaklyscat}
Let $C$ be a weakly scattered subset of a topological space $X$. Let $\{W_j\}_{j \in J}$ be a cover of $C$ by locally closed subsets, i.e.\ for any $j \in J$ there exist an open subset $V_j \subseteq X$ and a closed subset $Z_j \subseteq X$ such that $W_j=V_j \cap Z_j$.

Then, there exist an ordinal $\delta$ and an increasing sequence $\{U_{\leq \alpha}: \alpha \leq \delta \}$ of open subsets of $\overline{C}$ such that:
\begin{itemize}
\item[(i)] given $I_0$ the set of weakly isolated points of $C$, then $U_{\leq 0}$ consists in the disjoint union of subsets $U_{0, w}$ for $w \in I_0$ such that each $U_{0,w}$ is open in $\overline{C}$ and it is a neighbourhood of $w$;
\item[(ii)] given $I_{\alpha+1}$ the set of weakly isolated points of $C \setminus U_{\leq \alpha}$, then $U_{\leq \alpha+1}$ consists in the disjoint union of $U_{\leq \alpha}$ and subsets $U_{\alpha+1, w}$ for $w \in I_{\alpha+1}$ such that $U_{\alpha+1,w}$ is open in $\overline{C}\setminus U_{\leq \alpha}$ and it is a neighbourhood of $w$;
\item[(iii)] for a limit ordinal $\lambda$ we have $U_{\leq \lambda}=\bigcup_{\beta<\lambda}U_{\leq \beta}$;
\item[(iv)] $\{ U_{\alpha, w}: \alpha\leq \delta, w \in I_{\alpha}\}$ forms a cover of $C$ by locally closed subsets refining $\{W_j \}_{j \in J}$.
\end{itemize}
\end{proposition}
\begin{proof}
Without loss of generality, we can assume $X=\overline{C}$. We construct the invoked subsets $U_{\leq \alpha}$ and $U_{\alpha, w}$ using transfinite induction.

We first spell out the definition of $C$ being weakly scattered with respect to the topology induced by $\overline{C}$. For any closed subset $F\subseteq \overline{C}$ there exists a point $w \in F\cap C$ and an open subset $U \subseteq \overline{C}$ such that $w \in U\cap F\cap C \subseteq \overline{\{w \}}\cap C$.

For the starting step of the induction, we take $F=\overline{C}$ and set $I_0$ to be the collection of weakly isolated points in $C$. We saw that for a generic $w \in I_0$ we have an open neighbourhood $U_w$ such that $w \in U_w \cap C\subseteq \overline{\{ w\}}$. This means $U_w \cap \overline{\{ w\}}^c \cap C =\emptyset$, thus $U_w \cap \overline{\{ w\}}^c$ is an open subset of $\overline{C}\setminus C$. But $\overline{C}\setminus C$ has empty interior, due to $C$ being dense in $\overline{C}$, hence it follows $U_w \cap \overline{\{ w\}}^c=\emptyset$, i.e.\ $U_w \subseteq \overline{\{w \}}$. Now, since $\{ W_j\}$ is a cover of $C$ there must exist a $W_j$ such that $w \in W_j$. Suppose $W_j=V_j \cap Z_j$ with $V_j$ open and $Z_j$ closed respectively. We set $U_{0,w}=U_w \cap V_j$ and observe 
\[U_{0,w}=U_w \cap V_j\subseteq \overline{\{w\}}\cap V_j\subseteq Z_j \cap V_j=W_j. \]
We show that for two distinct $w, w' \in I_{0}$ the open neighbourhoods $U_{0,w}$ and $U_{0,w'}$ constructed with this method are disjoint. First, we prove that we cannot have inclusions $\overline{\{ w\}}\subseteq \overline{\{ w'\}}$ or $\overline{\{ w'\}}\subseteq \overline{\{ w\}}$. Indeed, we observe that $w$ and $w'$ must be weakly isolated points in $\overline{C}$, since we proved the open neighbourhoods $U_w$ and $U_{w'}$ are contained in the closures of the respective points. By \cite[Lemma~3.6]{spasublocales} it follows that $\overline{C}\setminus \overline{\{w \}}$ and $\overline{C}\setminus \overline{\{w' \}}$ are essential primes of $\emptyset$ in the frame of open subsets of $\overline{C}$. If it held  $\overline{\{ w\}}\subseteq \overline{\{ w'\}}$ we would have an inclusion $\overline{C}\setminus \overline{\{w' \}}\subseteq \overline{C}\setminus \overline{\{w\}}$, contradicting the fact that $\overline{C}\setminus \overline{\{w\}}$ is a minimal prime. The other inclusion $\overline{\{ w'\}}\subseteq \overline{\{ w\}}$ cannot hold by the same reasoning. Established this, we deduce that $w \not \in U_{w'}$, otherwise from $U_{w'}\subseteq \overline{\{w'\}}$ it would follow $\overline{\{ w\}}\subseteq \overline{\{ w'\}}$. Thus, $(U_{w'})^c$ is a closed subset containing $w$, hence we have the sequence of inclusions
\[ U_w\subseteq \overline{\{ w\}}\subseteq (U_{w'})^c. \]
Therefore, $U_w \cap U_{w'}=\emptyset$ and consequently $U_{0,w}\cap U_{0,w'}=\emptyset$. We conclude by setting $U_{\leq 0}=\bigcup_{w \in I_0}U_{0,w}$, which is clearly an open subset of $\overline{C}$.

Suppose the set $U_{\leq \alpha}$ has been defined, then we can construct $U_{\leq \alpha+1}$ with the same arguments we used for $U_{\leq 0}$. This time we take as ambient space the closed subset $\overline{C}\setminus U_{\leq \alpha}$ and $F=\overline{C}\setminus U_{\leq \alpha}$ when invoking $C$ being weakly scattered. Therefore, we set $I_{\alpha+1}$ to be the collection of weakly isolated points of $C \setminus U_{\leq \alpha}$ and conclude that for each $w \in I_{\alpha+1}$ there exists an open neighbourhood $U_{\alpha+1,w}$ contained in $W_j$ for some $j \in J$ such that for two different points $w$ and $w'$ the associated $U_{\alpha+1,w}$ and $U_{\alpha+1,w'}$ are disjoint. We set $U_{\leq \alpha+1}=U_{\leq\alpha}\cup \bigcup_{w \in I_{\alpha+1}}U_{\alpha+1, w}$ and observe that this is an open subset of $\overline{C}$ since each $U_{\alpha+1, w}$ is in the form $U_w \cap (U_{\leq \alpha})^c$ for some open subset $U_w \subseteq \overline{C}$.

If $\lambda$ is a limit ordinal we define $U_{\leq \lambda}=\bigcup_{\beta<\lambda}U_{\leq \beta}$ which is a union of open subsets.

We conclude by showing that this process must stop at some ordinal $\delta$. Suppose this is not the case, then for any ordinal $\alpha$ there exists a weakly isolated point in $C \cap U_{\leq \alpha+1}\setminus U_{\leq \alpha}$. By construction these points must be distinct. This would imply that the set $C$ contains a subset in correspondence with the collection of all ordinals, which is a proper class. Hence we get an absurd. 
\end{proof}

\begin{remark}\label{rmk-wscatpseudodim}
If $C$ is a scattered set, then we can take as neighbourhoods $U_{\alpha,w}=\{ w\}$ and $U_{\leq \alpha}$ coincides with $C_{\leq \alpha}$, the subset of $C$ given by the points with Cantor-Bendixson rank lesser or equal to $\alpha$.

The reverse is true: if $\{U_{\leq \alpha}\}$ induces a filtration of $C$ by a dimension function (in the sense of \cite[Def.~3.1]{ste-dimfunctions}) then this must coincide with the Cantor-Bendixson rank and $C$ must be scattered. Indeed, for any point $c \in U_{\alpha, w}$ we have $c \in \overline{\{w\}}$, hence by \cite[Def.~3.1 (ii)]{ste-dimfunctions} the dimension of $c$ is $\alpha$ and \cite[Def.~3.1 (iii)]{ste-dimfunctions} forces $c=w$. Consequently, $U_{\alpha,w}=\{w\}$ and an easy inductive argument shows $U_{\leq \alpha}=C_{\leq\alpha}$.
\end{remark}

\begin{remark}\label{rmk-liftloctoglobcovervialoc}
Let $Z\subseteq \text{Spc}(\CT^c)$ be the complement of a Thomason subset. We can form the localization adjunction $L_Z : \CT \rightleftarrows L_Z\CT : j$. We observe that $L_Z$ is a geometric morphism in the sense of \cite[Terminology~13.1]{BCHS-costratification}: in this case $f^*=L_Z$ and $f_*=j$.

In \cite[Rmk.~1.23]{BHS-stratification} it is explained that $\text{Spc}(L_Z)$ induces the identification of $\text{Spc}(L_Z\CT^c)$ with the subspace $Z\subseteq \text{Spc}(\CT^c)$.

Let $t \in \CT$ be such that $\text{Supp}(t)\subseteq Z$ and let $\{ W_i\}_{i \in I}$ be a cover of $\text{Supp}(t)$ by weakly visible subsets. Recall that by \cite[Rmk.~1.28]{BHS-stratification} we have $L_Zg(W_i)\cong g(W_i\cap Z)$.

Suppose we managed to prove that $L_Zt \in \text{Locid}\bc{t \otimes g(W_i \cap Z): i \in I}$ in the category $L_Z\CT$. Then we can use \cite[(13.4)]{BCHS-costratification} to deduce
\[ t\cong jL_Zt \in \text{Locid}\bc{t \otimes g(W_i\cap Z): i \in I}. \]
But since we have $g(W_i\cap Z)\cong L_Z\mbu \otimes g(W_i)$ we conclude
\[ t \in \text{Locid}\bc{t\otimes g(W_i): i \in I}. \]
Therefore, proving $t$ satisfies the local-to-global condition with respect to the cover $\{ W_i\}_{i \in I}$ is equivalent to proving it for $t$ considered as an object of the localized category $L_Z\CT$ taking as cover $\{W_i \cap Z \}_{i \in I}$.
\end{remark}

\begin{theorem}\label{thm-wscsuppimpliesltg}
Let $\emph{Supp}(t)\subseteq \emph{Spc}(\CT^c)$ be a weakly scattered subset with respect to the inverse topology. Then, the local-to-global principle holds for $t\in \CT$.
\end{theorem}
\begin{proof}
For sake of notation, we set $S=\text{Supp}(t)$. We start by considering a generic cover $\{W_j \}_{j \in J}$ of $\text{Spc}(\CT^c)$ by weakly visible subsets. By the detection property, if $W_j \cap S=\emptyset$ we have $t\otimes g(W_j)=0$, thus we can consider only the induced cover of $S$.

Using Remark~\ref{rmk-liftloctoglobcovervialoc}, it is enough to show the local-to-global principle holds for $t$ considered as an object of the category $L_{\overline{S}^{\vee}}\CT$ with respect to the cover $\{ W_j \cap \overline{S}^{\vee}\}$. Thus, from now on the $W_j$'s are subsets of $\overline{S}^{\vee}$.

In this setting, we apply Proposition~\ref{prop-refinementcoverweaklyscat} to the space $X=\text{Spc}(\CT^c)^{\vee}$ and its subset $C=S$. This gives us the subsets $\{ U_{\leq \alpha} \}$ and $\{ U_{\alpha, w}: w \in I_{\alpha} \}$ satisfying the properties (i), (ii), (iiii) and (iv). We recall that, by very definition of inverse topology, the open subsets of $\text{Spc}(\CT^c)^{\vee}$ are the Thomason subsets of $\text{Spc}(\CT^c)$.

We prove by transfinite induction that for any ordinal $\alpha$ it holds
\[ \Gamma_{U_{\leq \alpha}}t \in \text{Locid}\bc{t\otimes g(U_{\beta, w}): \beta \leq \alpha, w \in I_{\beta}}, \]
where $\Gamma_{U_{\leq \alpha}}$ denotes the colocalization functor associated with $(L_{\overline{S}^{\vee}}\CT)_{U_{\leq \alpha}}^c$. This exists since by construction $U_{\leq \alpha}$ is a Thomason subset of the  Balmer spectrum $\overline{S}^{\vee}\cong \text{Spc}(L_{\overline{S}^{\vee}}\CT^c)$.

If $\alpha=0$, by construction $U_{\leq 0}$ is a disjoint union of Thomason subsets $\{U_{0,w}: w \in I_0 \}$  of $\overline{S}^{\vee}$. Hence, we can use the argument of \cite[Lemma~3.13]{ste-dimfunctions} to prove that we have a decomposition
\[ \Gamma_{U_{\leq 0}}t \cong \coprod_{w \in I_0}\Gamma_{U_{0,w}}t \cong  \coprod_{w \in I_0}t \otimes g(U_{0,w})    \]
thus the claim holds.

Now suppose the claim has been proved for an ordinal $\alpha$, we show it must follow for the successor $\alpha+1$. We invoke the exact triangle
\[ \Gamma_{U_{\leq \alpha}}t\rightarrow \Gamma_{U_{\leq \alpha+1}}t\rightarrow L_{\overline{S}^{\vee}\setminus U_{\leq \alpha}}\Gamma_{U_{\leq \alpha+1}}t\rightarrow \Sigma \Gamma_{U_{\leq \alpha}}t. \]
If we prove that  $L_{\overline{S}^{\vee}\setminus U_{\leq \alpha}}\Gamma_{U_{\leq \alpha+1}}t\in \text{Locid}\bc{t\otimes g(U_{\alpha+1,w}): w \in I_{\alpha+1}}$ we are done. 

Again, we can use Remark~\ref{rmk-liftloctoglobcovervialoc} to reduce to showing the decomposition of this object as an object of the localizing subcategory $L_{\overline{S}^{\vee}\setminus U_{\leq \alpha}}\CT$. In this category, we have an isomorphism
\[ L_{\overline{S}^{\vee}\setminus U_{\leq \alpha}}\Gamma_{U_{\leq \alpha+1}}t\cong \Gamma_{U_{\leq \alpha+1}\setminus U_{\leq \alpha}}L_{\overline{S}^{\vee}\setminus U_{\leq \alpha}}t \]
where on the right-hand side we are invoking the colocalization functor associated with $U_{\leq \alpha+1}\setminus U_{\leq \alpha}$ which is a Thomason subset of $\overline{S}^{\vee}\setminus U_{\leq \alpha}$. By construction, $U_{\leq \alpha+1}\setminus U_{\leq \alpha}$ consists in the disjoint union of the Thomason subsets $U_{\alpha+1,w}$ for $w \in I_{\alpha+1}$. Thus, like in the case $\alpha=0$, we have a decomposition
\[ \Gamma_{U_{\leq \alpha+1}\setminus U_{\leq \alpha}}L_{\overline{S}^{\vee}\setminus U_{\leq \alpha}}t\cong \coprod_{w \in I_{\alpha+1}}t \otimes g(U_{\alpha+1, w}).\]
Finally, we consider the case when $\alpha$ is a limit ordinal $\lambda$. We can adapt the argument of \cite[Lemma~3.8]{ste-dimfunctions}, exchanging the subspaces $X_{\leq \alpha}$ in the reference for the subspaces $U_{\leq \beta}$ used here, to obtain the isomorphism
\[ \Gamma_{U_{\leq \lambda}}t\cong \hocolim_{\beta < \lambda}\Gamma_{U_{\leq \beta}}t.     \]
At this point, the claim easily follows from the inductive assumption: the fact that $\Gamma_{U_{\leq \beta}}t\in \text{Locid}\bc{t\otimes g(U_{\gamma, w}): \gamma\leq \beta, w \in I_{\gamma}}$ for all ordinals $\beta <\lambda$ implies
\[ \Gamma_{U_{\leq \lambda}}t\in \text{Locid}\langle t \otimes g(U_{\beta, w}): \beta<\lambda, w \in I_{\beta} \rangle=\text{Locid}\bc{t\otimes g(U_{\beta, w}): \beta \leq \lambda, w \in I_{\beta}}. \]
The claim taken for $\alpha=\delta$ together with $\{U_{\alpha,w} \}$ being a refinement of $\{ W_j\}$ gives the local-to-global principle for $t$.
\end{proof}

\begin{remark}
The proofs of \cite[Thm.~3.21]{BHS-stratification} and \cite[Thm.~7.6]{zou} have a motif different from the one of Theorem~\ref{thm-wscsuppimpliesltg}: they are non-constructive and rely on appropriate Rickard idempotents which are guaranteed to exist due to the topological conditions imposed on the ambient Balmer spectrum. 

In our situation, we are not making any assumption on the Balmer spectrum. Instead, we showed that the topological condition on the support of the fixed object is compatible with the topology of the Balmer spectrum in a way that allows us to refine the starting cover by weakly visible subsets. Then, we built explicitly the object in exam using a sequence of exact triangles and colimits involving Rickard idempotents with supports in the refinement. 

This approach was inspired by \cite{ste-dimfunctions}, with the caveat that in that paper the Balmer spectrum was filtered by a dimension function, while here
the filtration on the support of the fixed object which we constructed from the weakly scattered condition does not necessarily come from a dimension function. Recall Remark~\ref{rmk-wscatpseudodim}.
% the weakly scattered condition allows us to construct a filtration with weaker properties. Recall Remark~\ref{rmk-wscatpseudodim}.
%(flavour? concept? motif?, idea? look for right work) , where the local-to-global principle is deduced from the topological condition on the ambient Balmer spectrum, are non-constructive and rely on a yoga of exact triangles involving Rickard idempotents. Since we start from a generic object, which is not necessarily an idempotent,
\end{remark}

We now provide concrete applications of  Theorem~\ref{thm-wscsuppimpliesltg}.
\begin{lemma}\label{lem-proconstrclosurenoeth}
Let $K$ be a noetherian subset of a spectral space $X$. Then its constructible closure $\overline{K}^{\flat}$ is a noetherian set.
\end{lemma}
\begin{proof}
By \cite[Prop.~2.5]{ste-dimfunctions} we can assume $K$ is dense in $X^{\flat}$.

Assume $X$ is not noetherian, thus there exists a chain of open subsets
\[ U_1\subsetneq U_2\subsetneq U_3\subsetneq \dots \subsetneq U_n\subsetneq U_{n+1}\subsetneq \dots  \]
which does not stabilize. Without loss of generality, we can assume these $U_i$'s are open and quasi-compact. 

Taking the intersection with $K$ produces a chain of open subsets of $K$
\[ U_1\cap K\subseteq \dots \subseteq U_n \cap K\subseteq U_{n+1}\cap K\subseteq \dots \]
if we prove that $U_n \cap K\subsetneq U_{n+1}\cap K$ for every $n$ we obtain an absurd. The equality $U_n \cap K=U_{n+1}\cap K$ is equivalent to $U_{n+1}\cap U_n^c\subseteq K^c$. But $U_n$ and $U_{n+1}$ are clopen in the patch topology, hence $U_{n+1}\cap U_n^c$ is a non-empty open subset in this topology which is contained in $K^c$. This cannot be since $K$ is dense in $X^{\flat}$.
\end{proof}

\begin{lemma}\label{lem-noethisinvscat}
Let $K$ be a noetherian subset of a spectral space $X$. Then $K$ is scattered with respect to the inverse topology. 
\end{lemma}
\begin{proof}
By \cite[Prop.~2.5]{ste-dimfunctions} and Lemma~\ref{lem-proconstrclosurenoeth} we have $\overline{K}^{\flat}$ is a noetherian spectral space. \cite[Lemma~7.17]{supp-sanders} implies $\overline{K}^{\flat}$ is scattered in $X^{\vee}$. Therefore, $K$ is also scattered in $X^{\vee}$, being a subset of a scattered set.
\end{proof}

\begin{corollary}\label{cor-suppscatteredinprofinitesp}
Suppose $\emph{Supp}(t)\subseteq \emph{Spc}(\CT^c)$ is a noetherian subspace. Then the local-to-global principle holds for $t$.
\end{corollary}
\begin{proof}
Immediate from Theorem~\ref{thm-wscsuppimpliesltg} and Lemma~\ref{lem-noethisinvscat}.
\end{proof}

This corollary solves \cite[Question~21.8]{BCHS-costratification}. Theorem~\ref{thm-wscsuppimpliesltg} provides also a local version of \cite[Thm.~5.6]{ste-dimfunctions}.
\begin{corollary}
Suppose $\emph{Spc}(\CT^c)$ carries the constructible topology and $\emph{Supp}(t)\subseteq \emph{Spc}(\CT^c)$ is a scattered subset. Then the local-to-global principle holds for $t$.
\end{corollary}
\begin{proof}
It follows immediately from Theorem~\ref{thm-wscsuppimpliesltg} and the fact that on $\text{Spc}(\CT^c)$ the usual topology and the inverse topology coincide, since the Balmer spectrum has the constructible topology.
\end{proof}

\begin{remark}
We will adopt the notation of \cite[\S4.3]{spespa} regarding Cantor-Bendixson derivatives and the definition of perfect space.

We warn the reader that in this convention a perfect space is a space with no isolated points. A different choice, often found in the literature, is to call a space with no isolated points by \textit{dense-in-itself} and to use perfect to denote the closed dense-in-itself subset of a space, which coincides with the maximal dense-in-itself subset of the space.

Moreover, given a scattered set $C$ we denote by $C_{\leq \alpha}$ the subset of its elements with Cantor-Bendixson rank lesser or equal to $\alpha$. This follows the notation of \cite[Def.~5.1]{ste-dimfunctions}.
\end{remark}

\begin{example}\label{ex-orthogonalperfectsupport}
We assume here that the Balmer spectrum $\text{Spc}(\CT^c)$ is weakly noetherian. Suppose that $\CT$ does not satisfy the local-to-global principle. By \cite[Thm.~6.4]{BCHS-costratification}, this is equivalent to the subcategory
\[ \text{Locid}\bc{g(\fp): \fp \in \text{Spc}(\CT^c)}^{\perp} \] 
not being zero. We claim that the support of any non-zero object of such subcategory must be a perfect subspace of $\text{Spc}(\CT^c)^{\vee}$.

Indeed, let $t$ be one of such objects with $S=\text{Supp}(t)$. We work with the inverse topology and suppose the Cantor-Bendixson derivative $\delta S$ is a proper subset of $S$. This means that we have a decomposition $S=\delta S \sqcup S_{\leq 0}$ where $S_{\leq 0}$ is the set of isolated points of $S$. Then we can repeat the argument of Theorem~\ref{thm-wscsuppimpliesltg} to obtain the exact triangle
\[ \coprod_{\fp \in S_{\leq 0}}g(\fp)\otimes t \rightarrow t \rightarrow L_{\overline{S}^{\vee} \setminus S_{\leq 0}}t \rightarrow \Sigma  \coprod_{\fp \in S_{\leq 0}}g(\fp)\otimes t. \]
By the definition of $t$, the first morphism must be trivial, hence the triangle splits giving us
\[ L_{\overline{S}^{\vee} \setminus S_{\leq 0}}t \cong t \oplus \Sigma \coprod_{\fp \in S_{\leq 0}}g(\fp)\otimes t. \]
Observe that for any $\fp \in S_{\leq 0}$, this prime belongs to the support of the object on the right-hand side, while it does not belong to the support of the object on the left-hand side. This is an absurd, hence we conclude $S=\delta S$.
\end{example}

\section{Injective superdecomposable modules over non semi-artinian absolutely flat rings}
In this last section, we apply Corollary~\ref{cor-suppscatteredinprofinitesp} to the study of the support of injective superdecomposable modules over non semi-artinian absolutely flat rings. We refer the reader to \cite{ste-derabs} for the importance these objects have in the theory of tensor triangular geometry: they provide examples of objects not satisfying the local-to-global principle.

From now on, all rings which we consider will be unital and commutative.

Let $R$ be an absolutely flat ring which is not semi-artinian. The Balmer spectrum $\text{Spc}(\text{D}(R)^c)$ carries the constructible topology, hence it is weakly noetherian and the tensor triangular support coincides with the Balmer-Favi support. In \cite[Thm.~4.7]{ste-derabs} it is proved that any injective superdecomposable $R$-module is right orthogonal to the idempotents $g(\fp)$. Therefore, Example~\ref{ex-orthogonalperfectsupport} implies that its support must be a perfect subset of $\text{Spc}(\text{D}(R)^c)$.

We will present an injective superdecomposable $R$-module whose support coincides with the maximal perfect subset of $\text{Spc}(\text{D}(R)^c)$.

Before starting, we recall some basic facts which we will use through this section:
\begin{itemize}
\item[(a)] as consequence of \cite[Thm.~4.1]{thomason-classification}, there is a homeomorphism between the Zariski spectrum $\text{Spec}(A)$ of a ring $A$ and the Balmer spectrum $\text{Spc}(\text{D}(A)^c)$ given by 
\begin{align*}
\text{Spec}(A)&\cong \text{Spc}(\text{D}(A)^c) \\
P&\mapsto \mathfrak{p}=\{ M \in \text{D}(A)^c : M_P\simeq 0 \}
\end{align*}
where $M_P$ denotes the localization at the prime ideal $P$.

In the case of a general ring, one should be careful because this map is inclusion reversing. However, in the case of absolutely flat rings there are no proper inclusions of prime ideals. Therefore, from now on we will identify the Balmer spectrum of $\text{D}(R)$ with $\text{Spec}(R)$.
%Using this identification it is immediate that $\text{Spc}(\text{D}(R)^c)$ is $T_1$, hence carries the constructible topology (\cite[Prop.~1.3.20]{spespa}) as we mentioned earlier.

\item[(b)] For an absolutely flat ring $A$, given a generic element $a \in A$ it is possible to concoct an idempotent element $a' \in A$ such that $(a)=(a')$. Thus, all the principal ideals of $A$ are generated by idempotents. See \cite[Lemma~2]{olivier-absflatrings}.

\item[(c)] Given an ideal $I \subseteq A$ of an absolutely flat ring, the quotient $A/I$ is also absolutely flat. This guarantees that all the rings we will construct later will remain absolutely flat.
\end{itemize}

\begin{lemma} \label{lem-isolated<>simple}
Let $A$ be an absolutely flat ring. Then there is a bijection
\[ \{\emph{non-zero minimal ideals of} \ A\} \leftrightarrow \{\emph{isolated points of} \ \emph{Spec}(A)\}.
  \]
\end{lemma}
\begin{proof}
Suppose we start with a simple ideal $I$, this can be written as $(a)$ for a non-zero idempotent element $a\in A$. Thus $1-a$ is the orthogonal idempotent, we set $J=(1-a)$ and claim this is a maximal ideal. By construction we have $A=I+J$, now we show $I\cap J=0$. An element of $I\cap J$ can be written as $x=\lambda a=\mu(1-a)$, but using the fact that $a$ is idempotent we have $x=ax=\mu a(1-a)=0$. Hence we proved $A=I\oplus J$.

We now show $J$ is maximal. Suppose we have $J\subsetneq M$ for another ideal $M$, then taking the intersection of $M$ with $A$ we get $M=IM\oplus JM$. Since  $J\subsetneq M$ it must follow $IM\neq 0$, thus there exists $x\in M$ such that $ax\neq 0$. By minimality of $I$ it follows $(a)=(ax)$ hence $IM=I$. We have both $a$ and $1-a$ belong to $M$, thus $M=A$. This shows $J$ must be a maximal ideal.

By construction, we have $J\in U(a)$. If we prove $U(a)=\{J\}$ this proves the prime is isolated in the Zariski spectrum. Let $M$ be a maximal ideal different from $J$. By taking the intersection of $M$ with the decomposition $A=I\oplus J$ we get $M=IM\oplus JM$, if $IM=0$ we deduce $M\subseteq J$ contradicting the maximality of $M$. Hence $IM\neq 0$, this mean there exists $x\in M $ with $ax\neq 0$. By minimality of $I=(a)$, we deduce $(a)=(ax)\subseteq M$, thus $M\in V(a)$. Therefore the only prime ideal in $U(a)$ is $J$.

Suppose now we start from $P$ an isolated prime ideal. The Zariski spectrum has a basis given by $U(x)$ for $x\in A$. Since $U(x)\cap U(y)=U(xy)$, we have that $P$ is isolated if and only if $\{P\}=U(a)$ for some non-zero $a\in A$. Since $U(a)$ depends only on the ideal generated by $a$ and $A$ is absolutely flat, we can assume $a$ is idempotent. Observe that $R=(a)\oplus (1-a)$ as before. Moreover, $a(1-a)=0\in P$ and $P$ being prime implies $1-a\in P$. Thus $(1-a)\subseteq P$.

We now prove $(a)\cap P=0$. First observe that for any maximal ideal $M$ different from $P$ it holds $(a)\subseteq M$. Indeed, if $U(a)=\{P\}$ then $V(a)=\{M\in \text{Spec}(A) : M\neq P\}$. Therefore, we have
\[(a)\subseteq \bigcap_{M\neq P}M\]
and it follows
\[(a)\cap P\subseteq \bigcap_{M\in \text{Spec}(A)}M.\]
But the intersection of all prime ideals coincides with the nilradical of $A$, this ring is reduced being absolutely flat hence the nilradical is $0$. We conclude $(a)\cap P=0$.

Thus $(1-a)=P$, since $\forall p\in P$ we have $p=pa +p(1-a)=p(1-a)$. From the maximality of $P$ it follows $(a)$ is minimal. Indeed, suppose $I\subsetneq (a)$ is not zero, then $P+I$ is an ideal of $A$ which is strictly bigger than $P$, thus $I+P=A$. It follows $a\in I +P$, i.e.\ $a=ax+p$ with $ax\in I$ and $p\in P$. But $a=a^2=a^2x+ap=ax\in I$ and we conclude $I=(a)$. Thus $(a)$ is non-zero minimal.
\end{proof}

We give the following notation for the Loewy series of modules over a ring. We refer to \cite[Ch.~VIII \S2]{ringquot} for an exhaustive treatment.
\begin{definition}
Let $B$ be an arbitrary ring and let $M$ be a $B$-module. We set $sM$ to be the socle of $M$. Recall that the socle is defined to be the sum of all simple submodules of the module in question. We set $s_0M=0$ and $s_1M=sM$. Now we define recursively $s_{\alpha}M$, a submodule of $M$, for every ordinal $\alpha$. 

Suppose we defined $s_{\alpha}M$, then we set $s_{\alpha+1}M$ to be the submodule of $M$ such that $s_{\alpha+1}M/s_{\alpha}M$ is the socle of $M/s_{\alpha}M$. I.e.\ $s_{\alpha+1}M/s_{\alpha}M=s(M/s_{\alpha}M)$.

If $\lambda$ is a limit ordinal, then we set $s_{\lambda}M=\bigcup_{\alpha<\lambda}s_{\alpha}M$.

This sequence of submodules is called the \textit{Loewy series} of $M$.  The \textit{Loewy length} of $M$ is defined to be the minimal ordinal $\delta$ such that $s_{\delta}M=M$, if such an ordinal exists.

The ring $B$ being semi-artinian is equivalent to the existence of the Loewy length for $B$, considered as $B$-module. See \cite[Ch.~VIII \S2 Prop.~2.5]{ringquot}.
\end{definition}

\begin{lemma}\label{lem-socle<>CBder}
Let $A$ be an absolutely flat ring. Then we have
\[ \emph{Spec}(A/sA)\cong \{P\in \emph{Spec}(A): sA\subseteq P\}\cong \delta \emph{Spec}(A). \]
\end{lemma}
\begin{proof}
We have to show that $sA\subseteq P$ if and only if $P$ is not isolated in $\text{Spec}(A)$.

Suppose we have $P$ an isolated prime. Then in Lemma~\ref{lem-isolated<>simple} we saw there is a decomposition $A=P\oplus I$ with $I$ a simple module. Since $I\subseteq sA$ it follows $sA\not\subseteq P$.

Now suppose the prime $P$ does not contain $sA$. By definition, $sA=\sum_{I \ \text{simple}}I$ hence there must exit a simple module $I$ such that $I\not\subseteq P$. In Lemma~\ref{lem-isolated<>simple} we saw that $I=(a)$ with $a\in A$ idempotent and there is a decomposition $A=(a)\oplus (1-a)$. Again $P=(1-a)$ and $\{P\}=U(a)$, thus $P$ is isolated.
\end{proof}

This result allows us to produce a correspondence between Loewy series and Cantor-Bendixson derivatives.
\begin{proposition}\label{prop-specsocle=CBder}
Let $A$ be an absolutely flat ring. Then for any ordinal $\alpha$ we have
\[ \emph{Spec}(A/s_{\alpha}A)\cong \delta^{\alpha}\emph{Spec}(A).  \]
\end{proposition}
\begin{proof}
The argument is a basic transfinite induction, using Lemma~\ref{lem-socle<>CBder}.

If $\alpha=0$ the claim is trivial. Suppose we showed the claim for a general ordinal $\alpha$. Then we have
\begin{align*}
 \delta^{\alpha+1}\text{Spec}(A)&=\delta \delta^{\alpha}\text{Spec}(A)=\delta \text{Spec}(A/s_{\alpha}A)=\\
&=\text{Spec}\bigg( (A/s_{\alpha}A)/(s_{\alpha+1}A/s_{\alpha}A) \bigg)\cong \text{Spec}(A/s_{\alpha+1}A)  
\end{align*}
where the second equality comes from the induction assumption, the third relies on Lemma~\ref{lem-socle<>CBder} and the definition of $s_{\alpha+1}A$.

Now suppose $\lambda$ is a limit ordinal. By definition $\delta^{\lambda}\text{Spec}(A)=\bigcap_{\alpha<\lambda}\delta^{\alpha}\text{Spec}(A)$, using the inductive assumption we deduce $\delta^{\lambda}\text{Spec}(A)=\bigcap_{\alpha<\lambda}\text{Spec}(A/s_{\alpha}A)$. Observe that the filtration
\[ 0=s_0A\subseteq s_1A \subseteq \dots \subseteq s_{\alpha}A\subseteq s_{\alpha+1}A\subseteq \dots \subseteq s_{\lambda}A\subseteq \dots \]
induces the sequence of embeddings of spectral spaces
\[ \text{Spec}(A)\hookleftarrow \text{Spec}(A/s_1A)\hookleftarrow \dots \hookleftarrow \text{Spec}(A/s_{\alpha}A)\hookleftarrow \text{Spec}(A/s_{\alpha+1}A)\hookleftarrow \dots \hookleftarrow  \text{Spec}(A/s_{\lambda}A)\hookleftarrow \dots  \]
and recall we have the identification $\text{Spec}(A/s_{\alpha}A)\cong \{P\in \text{Spec}(A) : s_{\alpha}A\subseteq P\}$.

We conclude 
\begin{align*}
\bigcap_{\alpha<\lambda} \text{Spec}(A/s_{\alpha}A)& = \{ P\in \text{Spec}(A): s_{\alpha}A\subseteq P \ \forall \alpha<\lambda \} = \bigg\{ P\in \text{Spec}(A): \bigcup_{\alpha<\lambda} s_{\alpha}A \subseteq P \bigg\}=\\
&=\{ P\in \text{Spec}(A): s_{\lambda}A\subseteq P\}=\text{Spec}(A/s_{\lambda}A). 
\end{align*}
\end{proof}

The immediate consequence of this result is the following:
\begin{corollary}
Let $A$ be an absolutely flat ring. Then it is semi-artinian if and only if $\emph{Spec}(A)$ is scattered.
\end{corollary}

\begin{remark}
The correspondence of Proposition~\ref{prop-specsocle=CBder} is not new to the literature. Indeed, \cite[Cor.~5.3.60]{prest} implies that for an absolutely flat ring, under a technical condition, the Cantor-Bendixson rank and the Krull-Gabriel dimension (related to the Loewy series) on its Ziegler spectrum coincide. In this case, the Ziegler spectrum is homeomorphic to the Zariski spectrum by \cite[Thm.~8.2.91]{prest}. \cite[Prop.~8.2.87]{prest} gives sufficient conditions for the technical condition to be satisfied, however it does not guarantee the condition for a general non semi-artinian absolutely flat ring.

Thus, the complete proof of Proposition~\ref{prop-specsocle=CBder} is presented, to cover the non semi-artinian case. Moreover, its arguments will be used again in showing the identification of $s_{\alpha}A$ given in  Proposition~\ref{prop-soclefiltrarion=intersectionprimesCBfiltration} below.
\end{remark}

To conclude the correspondence between the Loewy series of an absolutely flat ring $A$ and the topology of its Zariski spectrum, we provide an explicit description of $s_{\alpha}A$ in terms of the prime ideals filtered by the Cantor-Bendixson rank.

\begin{lemma}\label{lem-socle=intersectionisolatedprimes}
Let $A$ be an absolutely flat ring. Then we have
\[ sA= \bigcap_{P \in \delta \emph{Spec}(A)}P. \]
\end{lemma}
\begin{proof}
We first prove that $Q \in \text{Spec}(A)$ is not isolated if and only if $\bigcap_{P \in \delta \text{Spec}(A)} P\subseteq Q$.

For the non-trivial implication we need to show $Q$ isolated implies $\bigcap_{P \in \delta \text{Spec}(A)} P\not \subseteq Q$.

Given $Q \in \text{Spec}(A)$ an isolated prime, we set $I\in A$ to be non-zero minimal ideal of $A$ associated with $Q$ by the correspondence of Lemma~\ref{lem-isolated<>simple}. Observe that for a prime ideal $P\neq Q$ we have $I \cap P$ is either $0$ or $I$, by the minimality of $I$. If this intersection were $0$, it would follow that $P \subseteq Q$, which is an absurd since $P$ is maximal. We conclude $P\cap I=I$, i.e.\ $I\subseteq P$.

Varying $P$ among all elements of $\delta \text{Spec}(A)$ we deduce $I \subseteq \bigcap_{P \in \delta \text{Spec}(A)}P$. If it held $\bigcap_{P \in \delta \text{Spec}(A)}P \subseteq Q$ we would get $I \subseteq Q$, an absurd since $A=Q\oplus I$ by construction. 

This equivalence together with Lemma~\ref{lem-socle<>CBder} implies that the subsets of $\text{Spec}(A)$ given by  $V\big(\bigcap_{P \in \delta \text{Spec}(A)}P \big)$ and $V(sA)$ coincide. This means that the radicals of the ideals $\bigcap_{P \in \delta \text{Spec}(A)}P$ and $sA$ coincide, but since $A$ is absolutely flat all ideals are radical. This proves the claim.
\end{proof}

\begin{proposition}\label{prop-soclefiltrarion=intersectionprimesCBfiltration}
Let $A$ be an absolutely flat ring. Then for any ordinal $\alpha$ it holds
\[ s_{\alpha}A=\bigcap_{P \in \delta^{\alpha}\emph{Spec}(A)}P. \]
\end{proposition}
\begin{proof}
We argue by transfinite induction on $\alpha$.

If $\alpha=0$ we have $s_0A=0$ by definition, while $\bigcap_{P \in \delta^0 \text{Spec}A}P=\bigcap_{P \in \text{Spec}A}P$. But the intersection of all maximal ideals of a ring coincides with the nilradical, in our case $A$ is reduced hence its nilradical is zero.

The case $\alpha=1$ is  Lemma~\ref{lem-socle=intersectionisolatedprimes}.

Now suppose the claim has been proved up to an ordinal $\alpha$, then we show the claim is true for $\alpha+1$.

Applying Lemma~\ref{lem-socle=intersectionisolatedprimes} to the ring $A/s_{\alpha}A$ we obtain
\[ s_{\alpha+1}A/s_{\alpha}A=s(A/s_{\alpha}R)=\bigcap_{Q \in \delta \text{Spec}(A/s_{\alpha}A)}Q. \]
Under the identification of Proposition~\ref{prop-specsocle=CBder} this set can be described as 
\[ \bigcap_{Q \in \delta \text{Spec}(A/s_{\alpha}A)}Q= \bigcap_{P \in \delta^{\alpha+1}\text{Spec}(A)}P/s_{\alpha}A. \]
Observe it holds the inclusion 
\[\Bigg( \bigcap_{P \in \delta^{\alpha+1}\text{Spec}(A)}P \Bigg)/s_{\alpha}A \subseteq \bigcap_{P \in \delta^{\alpha+1}\text{Spec}(A)}P/s_{\alpha}A =s_{\alpha+1}A/s_{\alpha}A. \]
However, a priori the first inclusion could be proper, hence we can only deduce $\bigcap_{P \in \delta^{\alpha+1}\text{Spec}(A)}P \subseteq s_{\alpha+1}A$.

Observe that the equality
\[ s_{\alpha+1}A/s_{\alpha}A=\bigcap_{P \in \delta^{\alpha+1}\text{Spec}(A)}P/s_{\alpha}A\]
implies $s_{\alpha+1}A/s_{\alpha}A\subseteq P/s_{\alpha}A$ for all $P \in \delta^{\alpha+1}\text{Spec}(A)$, thus $s_{\alpha+1}A\subseteq P$. This gives us the inverse inclusion $s_{\alpha+1}A\subseteq \bigcap_{P \in \delta^{\alpha+1}\text{Spec}(A)}P$ and we conclude these two ideals coincide.

Let $\lambda$ be a limit ordinal and suppose we showed the claim for all $\alpha<\lambda$. By definition $s_{\lambda}A=\bigcup_{\alpha<\lambda}s_{\alpha}A$ and we have $s_{\alpha}A=\bigcap_{P\in \delta^{\alpha}\text{Spec}(A)}P$. Observe that for $\beta <\gamma$ we have the inclusion $\bigcap_{P\in \delta^{\beta}\text{Spec}(A)}P\subseteq \bigcap_{P\in \delta^{\gamma}\text{Spec}(A)}P$. Thus, it follows that 
\[ \bigcup_{\alpha<\lambda}s_{\alpha}A=\bigcup_{\alpha<\lambda}\Bigg( \bigcap_{P\in \delta^{\alpha}\text{Spec}(A)}P\Bigg)= \colim_{\alpha<\lambda} \ \bigcap_{P\in \delta^{\alpha}\text{Spec}(A)}P=\bigcap_{P\in \bigcap_{\alpha<\lambda}\delta^{\alpha}\text{Spec}(A)}P=\bigcap_{P \in \delta^{\lambda}\text{Spec}(A)}P.  \]
This concludes the induction argument.
\end{proof}

\begin{corollary}\label{cor-perfectquotientringnotsemiart}
Let $R$ be an absolutely flat ring which is not semi-artinian. Let $\sigma$ be the minimal ordinal such that $s_{\sigma}R$ stabilizes. Hence, the ring $R/\bigcap_{P \in \delta^{\sigma}\emph{Spec}(R)}P$ is an absolutely flat ring with trivial socle and its Zariski spectrum corresponds to the maximal perfect subspace of $\emph{Spec}(R)$.
\end{corollary}
\begin{proof}
We have just to put together the results proved above. Proposition~\ref{prop-soclefiltrarion=intersectionprimesCBfiltration} states  $s_{\sigma}R=\bigcap_{P \in \delta^{\sigma}\text{Spec}(R)}P$ and Proposition~\ref{prop-specsocle=CBder} tells us that $\text{Spec}(R/s_{\sigma}R)$ coincides with $\delta^{\sigma}\text{Spec}(R)$. Since $s_{\sigma}R=s_{\sigma+1}R$ we have $\delta^{\sigma}\text{Spec}(R)=\delta^{\sigma+1}\text{Spec}(R)$, hence this subspace coincides with the maximal perfect subset of $\text{Spec}(R)$. Since the spectrum of $R/s_{\sigma}R$ has no isolated points Lemma~\ref{lem-isolated<>simple} implies that this ring has no minimal non-zero ideals.
\end{proof}
\begin{example}\label{ex-infiniteproductfields}
Let $k$ be any field and set $R=\prod_{\mathbb{N}}k$ which is an absolutely flat ring. In this case $\text{Spec}(R)$ is homeomorphic to $\beta \mathbb{N}$, the Stone-Cech compactification of the natural numbers. This space can be subdivided in its set of isolated points, corresponding to the principal ultrafilters on $\mathbb{N}$, and its maximal perfect subset, given by the non-principal ultrafilters.

For $i \in \mathbb{N}$, we set $e_i$ to be the element of $R$ with entry $1$ at its $i$-th coordinate and $0$ otherwise. Then $(1-e_i)$ is the isolated prime ideal corresponding to the principal ultrafilter induced by $i$. We have a decomposition $R=(1-e_i)\oplus (e_i)$ as in Lemma~\ref{lem-isolated<>simple}. It follows $s_1R=\oplus_i (e_i)=\sum_{\mathbb{N}}k$, coherently with the usual computations in the literature. The direct sum being the intersection of the prime ideals associated with the non-principal ultrafilters corresponds to the known fact that the cofinite filter coincides with the intersection of all non-principal ultrafilters.

In this situation, the ring invoked in Corollary~\ref{cor-perfectquotientringnotsemiart} is $\prod_{\mathbb{N}}k/ \sum_{\mathbb{N}}k$ which is known to have trivial socle.
\end{example}

We are finally ready to construct the injective superdecomposable module we wanted.

\begin{theorem}\label{thm-injsuperdecringnotsemiart}
Let $R$ be an absolutely flat ring which is not semi-artinian. Let $\sigma$ be the minimal ordinal such that $s_{\sigma}R$ stabilizes. Let $E$ denote the injective hull of $R/\bigcap_{P \in \delta^{\sigma}\emph{Spec}(R)}P$. Then $E$ is an injective superdecomposable module and as an element of $\emph{D}(R)$ its support coincides with $\delta^{\sigma}\emph{Spec}(R)$, the maximal perfect subset of $\emph{Spec}(R)$.
\end{theorem}
\begin{proof}
We first show $E$ is superdecomposable. By \cite[Lemma~5.14]{noncommnoethrings} this is equivalent to proving that $R/s_{\sigma}R$ does not admit uniform submodules, i.e.\ for any non-zero ideal $I \subseteq R/s_{\sigma}R$ there exist non-zero ideals $J_1, J_2\subseteq I$ such that $J_1 \cap J_2=0$.

Let us consider a generic non-zero ideal $I$ and take $i \in I \setminus \{0,1 \}$. Since we are interested only in the associated principal ideal $(i)\subseteq I$ we can assume $i$ is idempotent. Using the orthogonal idempotent $1-i$ we obtain the decomposition $R/s_{\sigma}R=(i)\oplus (1-i)$.

First consider the case $(1-i)\cap I=0$, which would imply $I=(i)$. We showed in Corollary~\ref{cor-perfectquotientringnotsemiart} that $R/s_{\sigma}R$ has no non-zero minimal ideals. Hence $(i)$ must have a proper non-zero submodule, say $(xi)\subsetneq (i)$ where again we can assume the element $x$ to be idempotent. The fact that $ix \neq i$ implies $i(1-x)\neq 0$. Therefore, $(ix)$ and $(i(1-x))$ provide two non-zero submodules of $I$ with trivial intersection.

Now consider the case $(1-i)\cap I \neq 0$, i.e.\ there exists some idempotent $x$ such that $(x(1-i))\subseteq I$ and $(x(1-i))\neq 0$. If it held $xi=0$, this would imply $0 \neq (x)\subseteq I$. In this situation the ideals $(x)$ and $(i)$ are non-zero submodules of $I$ such that $(i)\cap (x)=0$, just as we need. Thus, we assume $xi\neq 0$ and in this instance $(xi)$ and $(x(1-i))$ are the non-zero submodules of $I$ doing the job.

We finally show that $\text{Supp}(E)=\delta^{\sigma}\text{Spec}(R)$. We recall two important facts about the tensor-triangulated category $\text{D}(R)$. First, since all $R$-modules are flat, the tensor product of modules is an exact functor, thus the derived tensor product on $\text{D}(R)$ coincides with the non-derived tensor product of chains. Second, since the spectrum $\text{Spec}(R)$ is $T_1$ we have that the idempotent $g(P)$ associated with the prime $P$ coincides with $L_{P}\mbu=R_{P}$, the localization at $P$. Furthermore, it is easy to prove that the localization morphism $R\rightarrow R_P$ is surjective and its kernel coincides with $P$, thus we deduce $g(P)\cong R/P$.

We have that the support of the ring $R/s_{\sigma}R$ coincides with $\delta^{\sigma}\text{Spec}(R)$. Indeed, for any prime $Q \in \text{Spec}(R)$ we can compute
\[ R/{s_{\sigma}R}\otimes g(Q)= R/{s_{\sigma}R}\otimes R/Q=R/(s_{\sigma}R+Q). \]
Since $s_{\sigma}R=\bigcap_{P \in \delta^{\sigma}\text{Spec}(R)}P$, if $Q \in \delta^{\sigma}\text{Spec}(R)$ it follows $s_{\sigma}R+Q=Q$ and $R/{s_{\sigma}R}\otimes g(Q)\neq 0$. If instead $Q \not \in \delta^{\sigma}\text{Spec}(R)$, by Proposition~\ref{prop-soclefiltrarion=intersectionprimesCBfiltration} and Proposition~\ref{prop-specsocle=CBder} we have $\bigcap_{P \in \delta^{\sigma}\text{Spec}(R)}P\not \subset Q$. Thus $Q+\bigcap_{P \in \delta^{\sigma}\text{Spec}(R)}P$ is an ideal strictly containing the maximal ideal $Q$, hence it must coincide with the whole ring $R$ and consequently $R/{s_{\sigma}R}\otimes g(Q)=0$.

If we tensor the inclusion $R/s_{\sigma}R\hookrightarrow E$ with $g(P)$ we still have an inclusion since $g(P)$ is flat. Therefore, it follows that $\text{Supp}(R/s_{\sigma}R)\subseteq \text{Supp}(E)$. Example~\ref{ex-orthogonalperfectsupport} and \cite[Thm.~4.7]{ste-derabs} give us the inverse inclusion $\text{Supp}(E)\subseteq \delta^{\sigma}\text{Spec}(R)$ and we conclude.
\end{proof}

\begin{remark}
One could wonder if the ring $R/s_{\sigma}R$ is already injective, so that it is not necessary to form its injective hull. Unfortunately, this seems too optimistic: if we consider $R=\prod_{\mathbb{N}}k$ as in Example~\ref{ex-infiniteproductfields} then the main theorem of \cite{noninjectivecyclic} implies that $\prod_{\mathbb{N}}k/\sum_{\mathbb{N}}k$ is not an injective $R$-module.

Since this is one of the simplest examples of non semi-artinian absolutely flat rings, this should indicate we should not expect $R/s_{\sigma}R$ to be injective in general.
\end{remark}

\end{document}